\documentclass[review]{elsarticle} 
\usepackage[usenames]{color}
\usepackage{extarrows}
\usepackage{lineno,hyperref}
\modulolinenumbers[5]

\journal{..}
\usepackage{geometry}
\usepackage{amsfonts}
\usepackage{amsmath, cases}
\usepackage{indentfirst}
\usepackage{graphicx}
\usepackage{latexsym,bm, amsthm}

\usepackage{epsfig}
\usepackage{latexsym}
\usepackage{amssymb}
\allowdisplaybreaks[2]

\begin{document} 

\begin{frontmatter}

\title{Dual $r$-Rank Decomposition and Its Applications}

\author{Hongxing Wang}
\ead{winghongxing0902@163.com}

\author{Chong Cui}
\ead{cuichong0307@126.com}

\author{Xiaoji Liu$^{*}$}
\ead{xiaojiliu72@126.com}
\cortext[mycorrespondingauthor]{Corresponding author}
\address{School of Mathematics and Physics,
Guangxi Minzu University,
Nanning 530006, China}

\begin{abstract}
In this paper, we  introduce  the  dual $r$-rank decomposition of dual matrix, get its existence conditions and equivalent forms of the decomposition. Then we derive some characterizations of  dual Moore-Penrose generalized inverse(DMPGI). Based on DMPGI, we introduce  one special dual matrix(dual EP matrix). By applying the  dual $r$-rank decomposition, we derive several  characterizations of  dual EP matrix,  dual idempotent matrix, dual generalized inverses,  and relationships among dual Penrose equations.
\end{abstract}


\end{frontmatter}

\linenumbers

\section{Introduction}
\numberwithin{equation}{section}
\newtheorem{theorem}{T{\scriptsize HEOREM}}[section]
\newtheorem{lemma}[theorem]{L{\scriptsize  EMMA}}
\newtheorem{corollary}[theorem]{C{\scriptsize OROLLARY}}
\newtheorem{proposition}[theorem]{P{\scriptsize ROPOSITION}}
\newtheorem{remark}{R{\scriptsize  EMARK}}[section]
\newtheorem{definition}{D{\scriptsize  EFINITION}}[section]
\newtheorem{algorithm}{A{\scriptsize  LGORITHM}}[section]
\newtheorem{example}{E{\scriptsize  XAMPLE}}[section]
\newtheorem{problem}{P{\scriptsize  ROBLEM}}[section]
\newtheorem{assumption}{A{\scriptsize  SSUMPTION}}[section]

Clifford firstly proposed the dual number {\cite{Clifford1873}} in 1873,
then  Study  \cite{Study1891}  gave its specific form.
Subsequently,
the dual algebra has developed rapidly and been widely applied to
dynamic analysis of spatial mechanisms, sensor calibration, robotics and other fields (see {\cite{Falco2018,Pennestr2018,Udwadia2020,Yang2010}}).
In recent years, some researches of dual matrix,
dual  generalized inverse,
 dual equation and their applications
 have further promoted the development of dual algebra theory and its applications (see {\cite{Angeles2012, Belzile2019, Condurache2012,Falco2018,Gutin2021lma,Pennestr2009}}).

In this paper, we adopt the following notations:
$\mathbb{R}_{m\times n}$ stands for the set of all $m\times n$ real matrices;
 ${\rm rk}(A)$ for the rank of $A$;
$Q_{A,B}^{S}$ for $A^{T}B+B^{T}A$.
Let the dual number be $\widehat{a}$ and have the following form:
$$\widehat{a}=a+\epsilon a^{\circ}$$
in which $a$ and $a^{\circ}$ are real numbers,
and $\epsilon$ is the dual unit subjected to the rules
$$\epsilon \neq 0, \ 0\epsilon=\epsilon 0=0,  \
1\epsilon=\epsilon 1=\epsilon \mbox{\ \  and \ } \epsilon^{2}=0.$$
If a matrix has the form of $A_0+\epsilon A_1$ and
$A_i\in\mathbb{R}_{m\times n}(i=0,1)$,
it can be called the dual matrix and denoted as $\widehat{A}$.
Furthermore,
denote the set of all $m\times n$ dual matrices as $\mathbb{D}_{m\times n}$.

The dual Moore-Penrose generalized inverse(DMPGI for short) of $\widehat{A}$ is the unique dual matrix $\widehat{X}$,
which satisfies  the following four dual Penrose equations {\cite{Pennestr2009}}:
\begin{align}
\label{20220216-8}
\left(\widehat{1}\right)\ \widehat{A}\widehat{X}\widehat{A}=\widehat{A},
\left(\widehat{2}\right)\ \widehat{X}\widehat{A}\widehat{X}=\widehat{A},
\left(\widehat{3}\right)\ \left(\widehat{A}\widehat{X}\right)^{T}=\widehat{A}\widehat{X},
\left(\widehat{4}\right)\ \left(\widehat{X}\widehat{A}\right)^{T}=\widehat{X}\widehat{A},
\end{align}
and the unique dual matrix $\widehat{X}$ is denoted by $\widehat{X}=\widehat{A}^{\dag}$.
Especially,
the DMPGI has expanded  the application range of the generalized inverse theory.
It is worth noting that,  unlike real matrix,
dual matrix may not have DMPGI.
When $A_1=0$,
 $\widehat{A}=A_0$ is a real matrix,
 then
the Moore-Penrose generalized inverse of $A_0$ is the unique matrix $X$ satisfying the following four Penrose equations:
$$(1)\ A_0XA_0=A_0, \ (2)\ XA_0X=X,
\ (3)\ (A_0X)^{T}=A_0X, \  (4)\ (XA_0)^{T}=XA_0$$
and the unique matrix $X$ is denoted by $X=A_0^{\dag}$.
 Let $A_0\{i,...,k\}$ denote the set of solutions
 which satisfy equations $(i),..., (k)$ from the above four Penrose equations (1)-(4).
Therefore  $X$ can be called $\{i,...,k\}$-inverse of $A_0$,
and   denoted by $A_0^{(i,...,k)}$(see {\cite{Adi2003}}).
It is well known that
a variety of   generalized inverses,
such as Drazin inverse,
group inverse,
core inverse and core-EP inverse,
 have been established successively.
The achievements of generalized inverse theory have been greatly enriched,
and the scope of their applications has been expanded to physics, statistics, etc.
For more information about generalized inverse theory and its applications,
please refer to  {\cite{Adi2003,Rao1973,Wang2018}}.

Full-rank decomposition is one of the basic decompositions in matrix theory.
It has the following definitions {\cite{Adi2003,Zhang1999}}:
Let $A\in\mathbb{R}_{m\times n}$ and
${\rm rk}(A)=r$,
then there exist full column rank matrix $F\in\mathbb{R}_{m\times r}$
and full row rank matrix $G\in\mathbb{R}_{r\times n}$
such that $A=FG$.
Not only does full rank decomposition play an important role in solving generalized inverse matrix,
but also has a wide range of applications in many fields such as   mathematical statistics, systems theory, optimization and cybernetics. 
For example,
the full rank decomposition can be used to represent the $\{i,...,k\}$-inverse of matrix $A$ {\cite{Adi2003}}:
Let $A\in \mathbb{R}_{m\times n}$, ${\rm rk}(A)=r$,
and its full rank decomposition is $A=FG$,  in which  ${\rm rk}(F)= {\rm rk}(G)=r$,
then
\begin{align}
\label{20220216-2}
A^\dag&=G^{\dag}F^{\dag},\
G^{\dag}  =G^{T}\left(GG^{T}\right)^{-1},\
F^{\dag} =\left(F^{T}F\right)^{-1}F^{T},
\\
\label{20220216-3}
G^{(i)}F^{(1)}& \in A{\{i\}},  i=1,2,4 \mbox{\ and \ } G^{(1)}F^{(j)}\in A{\{j\}},  j=1,2,3.
\end{align}
For more details, please refer to   literatures {\cite{Adi2003,Rao1973}}.

In this paper, we   extend the full rank decomposition from real matrix to dual matrix, 
introduce the  dual $r$-rank decomposition,
get some equivalent characterizations of the existence of dual $r$-rank decomposition,
and
give  a method of calculating dual $r$-rank decomposition.
By applying the
decomposition,
we get characterizations of DMPGI and  relationships among dual Penrose equations. 
Furthermore,
we give  a method of calculating DMPGI
and some examples.
In addition,
we consider two special dual matrices: dual EP matrix and dual idempotent matrix.
We give the definition of  dual EP matrix,
and
get characterizations and dual $r$-rank decompositions of  both dual EP matrix and dual idempotent matrix.
At last, by applying the dual $r$-rank decomposition and definitions  of these special dual matrix,
we get  characterizations of both DMPGIs of  dual EP matrix and dual idempotent matrix.

\section{Preliminaries}
This section provides several  results that will be used in the following sections.

\begin{lemma}[{\cite{Wang2021}}]
Let $\widehat{A}\in\mathbb{D}_{m\times n}$
and  $\widehat{A}= A_0+\epsilon A_1$.
Then
the DMPGI of $\widehat{A}$ exists if and only if
\begin{align}
\label{1.1}
\left(I_m-A_0A_0^{\dag}\right)A_1\left(I_n-A_0^{\dag}A_0\right)=0.
\end{align}
Furthermore,
\begin{align}
\nonumber
\widehat{A}^{\dag}
=
A_0^{\dag}
-\varepsilon
 \left(A_0^{\dag}A_1A_0^{\dag}\right.
 &
 -\left(A_0^{T}A_0\right)^{\dag}A_1^{T}\left(I_m-A_0A_0^{\dag}\right)
 \\
\label{The-DMPGI}
&
 -\left.\left(I_n-A_0^{\dag}A_0\right)A_1^{T}\left(A_0A_0^{T}\right)^{\dag}\right).
\end{align}
\end{lemma}

\begin{lemma}[{\cite{Udwadia2021}}]
\label{2.5}
Let
$\widehat{A_1}\in\mathbb{D}_{m\times r}$,
$\widehat{A_2}\in\mathbb{D}_{r\times n}$,
$\widehat{A_1}= A_2 + \epsilon A_3$,
$\widehat{A_2}= A_4 + \epsilon A_5$,
${\rm rk}(A_2)=r$ and ${\rm rk}(A_4)=r$.
Then
\begin{align}
\label{2.20-1}
\widehat{A_1}^{\dag}&=\left(\widehat{A_1}^{T}\widehat{A_1}\right)^{-1}\widehat{A_1}^{T}
\\
\label{2.20}
&=\left(A_2^{T}A_2\right)^{-1}A_2^{T}+\epsilon
\left(\left(A_2^{T}A_2\right)^{-1}A_3^{T}-\left(A_2^{T}A_2\right)^{-1}Q_{A_2,A_3}^{S}\left(A_2^{T}A_2\right)^{-1}A_2^{T}\right)
\end{align}
and
\begin{align}
\label{2.20-2}
\widehat{A_2}^{\dag}&=\widehat{A_2}^{T}\left(\widehat{A_2}\widehat{A_2}^{T}\right)^{-1}
\\
\label{2.19}
&=A_4^{T}\left(A_4A_4^{T}\right)^{-1}+\epsilon
\left(A_5^{T}\left(A_4A_4^{T}\right)^{-1}-A_4^{T}\left(A_4A_4^{T}\right)^{-1}Q_{A_4^{T},A_5^{T}}^{S}\left(A_4A_4^{T}\right)^{-1}\right),
\end{align}
where $Q_{A_2,A_3}^{S}=A_2^{T}A_3+A_3^{T}A_2$
and $Q_{A_4^{T},A_5^{T}}^{S}=A_4A_5^{T}+A_5A_4^{T}$ .
\end{lemma}

\begin{lemma}[{\cite{Liu2006}}]
\label{2.1}
Let
 $A\in\mathbb{R}_{m\times p}$,
$B\in\mathbb{R}_{q\times n}$
and $ C\in\mathbb{R}_{m\times n}$.
Then the matrix equation
\begin{align}
\label{2.11}
AX+YB=C
\end{align}
is consistent if and only if
\begin{align}
\label{2.12}
\left(I_{m}-AA^{\dag}\right)C\left(I_{n}-B^{\dag}B\right)=0.
\end{align}
then
  the solution of this equation is
\begin{subnumcases}{}
\label{2.13}
X=A^\dag C+UB+\left(I_p-A^\dag A\right)V,
\\
\label{2.14}
Y=\left(I_m-AA^{\dag}\right)CB^{\dag}-AU+W\left(I_q-BB^{\dag}\right),
\end{subnumcases}
where $U\in \mathbb{R}_{p\times q}$, $V\in\mathbb{R}_{p\times n}$ and $W\in\mathbb{R}_{m\times q}$ are arbitrary.
\end{lemma}

\section{Dual $r$-rank Decomposition}

In this section we extend the full rank decomposition of real matrix to dual matrix. 
We also give the definitions of $r$-row full rank dual matrix, $r$-column full rank    dual matrix,
and  dual $r$-rank decomposition.
Furthermore,
we give  characterizations of the existence of the dual $r$-rank decomposition, a  method of calculating the decomposition,
and two  examples.

\begin{definition}
Let
$\widehat{A_1}\in\mathbb{D}_{m\times r}$,
$\widehat{A_2}\in\mathbb{D}_{r\times n}$,
$\widehat{A_1}= A_2 + \epsilon A_3$
and
$\widehat{A_2}= A_4 + \epsilon A_5$.
If the real part matrix $A_2$ of $\widehat{A_1}$ is a column full rank matrix,
then we call  $\widehat{A_1}$  $r$-column full rank dual matrix;
if the real part matrix $A_4$ of $\widehat{A_2}$ is a row full rank  matrix,
then  we call  $\widehat{A_2}$   $r$-row full rank dual matrix.
\end{definition}

\begin{definition} {\bf (Dual $r$-rank Decomposition)}
\label{2.111}
Let
 $\widehat{A}\in\mathbb{D}_{m\times n}$,
 $\widehat{A}= A_0+\epsilon A_1$,
${\rm rk}(A_0)=r$,
and
$A_0=A_2A_4$ be a full rank decomposition of $A_0$.
If there exist
an $r$-column full rank  dual matrix $\widehat{A_1}= A_2+\epsilon A_3$
and
an $r$-row full rank dual  matrix $\widehat{A_2}= A_4+ \epsilon A_5$,
such that
$$\widehat{A}= \widehat{A_1}\widehat{A_2},$$
which we call a dual $r$-rank decomposition of $\widehat{A}$.
\end{definition}

From Definition \ref{2.111},
  the following results can be inferred.

\begin{theorem}
\label{2.2}
Let
 $\widehat{A}\in\mathbb{D}_{m\times n}$,
 $\widehat{A}= A_0+\epsilon A_1$,
${\rm rk}(A_0)=r$,
and $A_0=A_2A_4$ be a full rank decomposition of $A_0$.
Then the dual $r$-rank decomposition
 of $\widehat{A}$ exists if and only if
\begin{align}
\label{2.15}
\left(I_{m}-A_2A_2^{\dag}\right)A_1\left(I_{n}-A_4^{\dag}A_4\right)=0.
\end{align}
Furthermore,
if $\widehat{A}$ has a dual $r$-rank decomposition $\widehat{A}= \widehat{A_1}\widehat{A_2}$,
in which
$ \widehat{A_1}= A_2+\epsilon A_3$ and $ \widehat{A_2}= A_4+ \epsilon A_5$,
then
\begin{align}
\label{20210210-1}
\left\{
\begin{aligned}
A_3&=\left(I_m-A_2A_2^{\dag}\right)A_1A_4^{\dag}-A_2P,
\\
A_5&=A_2^{\dag}A_1+PA_4,
\end{aligned}
\right.
\end{align}
for arbitrary
 $P\in \mathbb{R}_{r\times r}$. 
\end{theorem}

\begin{proof}
$"\Rightarrow"$:
Suppose the dual $r$-rank decomposition of the dual matrix $\widehat{A}$ exists.
Let $\widehat{A}= \widehat{A_1}\widehat{A_2}$ be a dual $r$-rank decomposition of $\widehat{A}$,
where
$$\widehat{A_1}=A_2+\epsilon Y \mbox{\ \  and \ } \widehat{A_2}=A_4+ \epsilon X.$$
Then $A_0+\epsilon A_1=(A_2+\epsilon Y)(A_4+\epsilon X)$.
By expanding this equation, 
we have
\begin{align}
\label{2.16}
A_2X+YA_4=A_1.
\end{align}
By applying Lemma {\ref{2.1}} to the equation(\ref{2.16}),
we get   ({\ref{2.15}}).

$"\Leftarrow"$:
Let  $A_0=A_2A_4$ be a full rank decomposition of $A_0$.
Because ({\ref{2.15}}) holds,
by applying Lemma {\ref{2.1}},
we  get that the equation $A_2X+YA_4=A_1$ is consistent,
and the solution to this equation is
\begin{align}
\label{20210210-2}
\left\{
\begin{aligned}
X&=A_2^{\dag}A_1+PA_4,
\\
Y&=\left(I_m-A_2A_2^{\dag}\right)A_1A_4^{\dag}-A_2P,
\end{aligned}
\right.
\end{align}
for arbitrary
 $P\in \mathbb{R}_{r\times r}$.
Let $\widehat{A_1}=A_2+\epsilon Y$ and $ \widehat{A_2}=A_4+\epsilon X$.
Then
$\widehat{A_1}=A_2+\epsilon Y$ is an $r$-column full rank  dual matrix;
$\widehat{A_2}=A_4+\epsilon X$ is an $r$-row full rank  dual matrix;
$$\widehat{A_1}\widehat{A_2}=(A_2+\epsilon Y)(A_4+\epsilon X)
=
A_2A_4+\epsilon(A_2X+YA_4)=A_0+\epsilon A_1=\widehat{A}.$$
Therefore,  the dual $r$-rank decomposition of $\widehat{A}$ exists.

In summary,
the dual $r$-rank decomposition of $\widehat{A}$ exists if and only if the equation (\ref{2.15}) is consistent. Furthermore, by applying (\ref{20210210-2}), we get  (\ref{20210210-1}).
\end{proof}

Based on Theorem {\ref{2.2}},
the detailed calculation process of dual $r$-rank decomposition is given as follows,
and corresponding examples are also given to verify this process.

\textbf{(1).}
Input matrix $A_0$ and $A_1$, and the form of  dual matrix $\widehat{A}$ is
$\widehat{A}=A_0+\epsilon A_1, \ A_i \in \mathbb{R}_{m\times n}, \ {\rm rk}(A_0)=r$;

\textbf{(2).} Perform full rank decomposition on $A_0$: $A_0=A_2 A_4$, 
in which  $A_2$ is  a column full rank matrix and
$A_4$   is  a row full rank matrix;

\textbf{(3).}
Calculate the Moore-Penrose inverses of   $A_2$ and   $A_4$: $A_2^{\dag}$ and $A_4^{\dag}$;

\textbf{(4).} Check whether the matrix equation $A_2X+YA_4=A_1$ is consistent:
$$\left(I_m-A_2A_2^{\dag}\right)A_1\left(I_n-A_4^{\dag}A_4\right)=0.$$
 If the matrix equation holds,
then proceed to step (5);

\textbf{(5).} Calculate the solution to matrix equation $A_2X+YA_4=A_1$:
\begin{align}
\nonumber
\left\{
\begin{aligned}
X&=A_2^{\dag}A_1+PA_4,
\\
Y&=\left(I_m-A_2A_2^{\dag}\right)A_1A_4^{\dag}-A_2P,
\end{aligned}
\right.
\end{align}
  where $P$ is arbitrary;

\textbf{(6).} Get one dual $r$-rank decomposition of the dual matrix $\widehat{A}$:
$\widehat{A}=\widehat{A_1}\widehat{A_2}=(A_2+\epsilon A_3)(A_4+\epsilon A_5)$.

\begin{example}
Let \begin{align}
\nonumber
\widehat{A}=A_0+\epsilon A_1=
\left[\begin{matrix}
    1&   0\\
    0&   0
\end{matrix}\right]+
\epsilon\left[\begin{matrix}
    1&   1\\
    1&   1
\end{matrix}\right].
\end{align}

 By performing full rank decomposition of $A_0=A_2A_4$
where
\begin{align}
\nonumber
A_2=\left[\begin{matrix}
    1\\
    0
\end{matrix}\right] \mbox{\ \  and \ }
A_4=
\left[\begin{matrix}
1&   0
\end{matrix}\right],
\end{align}
we have
\begin{align}
\nonumber
A_2^{\dag}=\left[\begin{matrix}
    1&  0
\end{matrix}\right] \mbox{\ \  and \ }
A_4^{\dag}=
\left[\begin{matrix}
   1\\
   0
\end{matrix}\right],
\end{align}
and by calculating $\left(I_2-A_2A_2^{\dag}\right)A_1\left(I_2-A_4^{\dag}A_4\right)$,
we can get
\begin{align}
\nonumber
\left(I_2-A_2A_2^{\dag}\right)A_1\left(I_2-A_4^{\dag}A_4\right)
&=\left(\left[\begin{matrix}
   1&   0\\
   0&   1
\end{matrix}\right]-
\left[\begin{matrix}
   1\\
   0
\end{matrix}\right]
\left[\begin{matrix}
    1&  0
\end{matrix}\right]\right)
\left[\begin{matrix}
    1&   1\\
    1&   1
\end{matrix}\right]
\left(\left[\begin{matrix}
   1&   0\\
   0&   1
\end{matrix}\right]-
\left[\begin{matrix}
   1\\
   0
\end{matrix}\right]
\left[\begin{matrix}
    1&  0
\end{matrix}\right]\right)
\\
\nonumber
&=\left[\begin{matrix}
   0&   0\\
   0&   1
\end{matrix}\right]\neq 0.
\end{align}
By applying Theorem \ref{2.2},
we know that
 $\widehat{A}$ does not have the dual $r$-rank decomposition.
\end{example}

\begin{example}
\label{2.3}
Calculate the dual $r$-rank decomposition of
\begin{align}
\nonumber
\widehat{A}=A_0+\epsilon A_1=
\left[\begin{matrix}
    1&   2&  1\\
    2&   1&  1\\
    3&   3&  2
\end{matrix}\right]+
\epsilon\left[\begin{matrix}
    1&   4&  7\\
    2&   5&  8\\
    3&   6&  14
\end{matrix}\right].
\end{align}

The rank of matrix $A_0$ is ${\rm rk}(A_0)=2$.
 By performing full rank decomposition of $A_0=A_2A_4$
where
\begin{align}
\nonumber
A_2=\left[\begin{matrix}
    1&   2 \\
    2&   1 \\
    3&   3
\end{matrix}\right] \mbox{\ \  and \ }
A_4=
\left[\begin{matrix}
1&   0&  \frac{1}{3}\\
0&   1&  \frac{1}{3}
\end{matrix}\right],
\end{align}
we have
\begin{align}
\nonumber
A_2^{\dag}=\left[\begin{matrix}
-\frac{4}{9}&     \frac{5}{9}&    \frac{1}{9}\\
 \frac{5}{9}&    -\frac{4}{9}&    \frac{1}{9}
\end{matrix}\right] \mbox{\ \  and \ }
A_4^{\dag}=
\left[\begin{matrix}
\frac{10}{11}&   -\frac{1}{11}\\
-\frac{1}{11}&    \frac{10}{11}\\
\frac{3}{11}&     \frac{3}{11}
\end{matrix}\right].
\end{align}
It is easy to check that
 $\left(I_3-A_2A_2^{\dag}\right)A_1\left(I_3-A_4^{\dag}A_4\right)=0$.
Therefore,  the matrix equation {\rm(\ref{2.16})} is consistent.
Let \begin{align}
\nonumber
P=\left[\begin{matrix}
  \frac{1}{2}&  \frac{1}{2}\\
   -1&          \frac{1}{2}
\end{matrix}\right].
\end{align}
Then   the solution to {\rm(\ref{2.16})} is
\begin{align}
\nonumber
\left\{
\begin{aligned}
X&=A_2^{\dag}A_1+PA_4
=\left[\begin{matrix}
\frac{3}{2}&   \frac{13}{6}&    \frac{29}{9}\\
    -1&        \frac{7}{6}&     \frac{31}{18}
\end{matrix}\right],
\\
Y&=\left(I_3-A_2A_2^{\dag}\right)A_1A_4^{\dag}-A_2P
=\left[\begin{matrix}
\frac{1}{2}&     -\frac{1}{2}\\
      0&         -\frac{1}{2}\\
 \frac{3}{2}&          -4
\end{matrix}\right],
\end{aligned}
\right.
\end{align}

 Let $X=A_5$ and $Y=A_3$,
then we can get
\begin{align}
\nonumber
\left\{
\begin{aligned}
\widehat{A_1}&=A_2+\epsilon A_3=
\left[\begin{matrix}
    1&   2\\
    2&   1\\
    3&   3
\end{matrix}\right]
+\epsilon
\left[\begin{matrix}
\frac{1}{2}&     -\frac{1}{2}\\
      0&         -\frac{1}{2}\\
 \frac{3}{2}&          -4
\end{matrix}\right],
\\
\widehat{A_2}&=A_4+\epsilon A_5=
\left[\begin{matrix}
1&   0&  \frac{1}{3}\\
0&   1&  \frac{1}{3}
\end{matrix}\right]
+\epsilon
\left[\begin{matrix}
\frac{3}{2}&   \frac{13}{6}&    \frac{29}{9}\\
    -1&        \frac{7}{6}&     \frac{31}{18}
\end{matrix}\right].
\end{aligned}
\right.
\end{align}
Next we verify that $\widehat{A}=\widehat{A_1}\widehat{A_2}$ is a dual $r$-rank decomposition of $\widehat{A}$.
Multiplying $\widehat{A_1}$ by $\widehat{A_2}$ gives
\begin{align}
\nonumber
\widehat{A_1}\widehat{A_2}&=(A_2+\epsilon A_3)(A_4+\epsilon A_5)=
A_2A_4+\epsilon A_2A_5+\epsilon A_3A_4=A_0+\epsilon A_1
\\&
\nonumber
=\left[\begin{matrix}
    1&   2&  1\\
    2&   1&  1\\
    3&   3&  2
\end{matrix}\right]+
\epsilon\left[\begin{matrix}
    1&   4&  7\\
    2&   5&  8\\
    3&   6&  14
\end{matrix}\right].
\end{align}
Hence,
 $\widehat{A}=\widehat{A_1}\widehat{A_2}$ is a dual $r$-rank decomposition of $\widehat{A}$.
\end{example}

\begin{remark}
Since the full rank decomposition of the real part matrix $A_0$ of $\widehat{A}$ is not unique,
  the solutions $X$ and $Y$ to the matrix equation(\ref{2.16}) are not unique.
Let $P$ is a zero matrix.
By applying Theorem \ref{2.2},
it is  obvious that 
$A_2+\epsilon\left(I_m-A_2A_2^{\dag}\right)A_1A_4^{\dag}$
is an $r$-column full rank  dual matrix;
$A_4+\epsilon A_2^{\dag}A_1$
is  an $r$-row full rank dual  matrix;
\begin{align}
\label{20220216-10}
 \widehat{A}
=\left(A_2+\epsilon\left(I_m-A_2A_2^{\dag}\right)A_1A_4^{\dag}\right)\left(A_4+\epsilon A_2^{\dag}A_1\right).
\end{align}
Therefore,
(\ref{20220216-10}) is one
 dual $r$-rank decomposition of $\widehat{A}$.
 \end{remark}

\section{Applications of Dual $r$-rank Decomposition}
In this section,
we apply dual $r$-rank decomposition to studying several related problems,
 including characterization and calculation of DMPGI,  special dual matrices and their properties, 
 and dual Penrose equations.

\subsection{Dual Moore-Penrose Generalized Inverse}
Let $A_0 \in \mathbb{R}_{m\times n}, \ {\rm rk}(A_0)=r$,
and $A_0=A_2A_4$ be a full rank decomposition of $A_0$.
It is well known that
\begin{align}
\label{20220216-7}
A_0A_0^{\dag}=A_2A_2^{\dag} \mbox{\ \  and \ } A_0^{\dag}A_0=A_4^{\dag}A_4.
\end{align}
By using (\ref{20220216-7}), we can get the following Theorems.

\begin{theorem}
\label{2.4}
Let
 $\widehat{A}\in\mathbb{D}_{m\times n}$,
 $\widehat{A}= A_0+\epsilon A_1$
and
${\rm rk}(A_0)=r$.
Then the following conditions are equivalent:

{\rm (a)}.  the dual $r$-rank decomposition of $\widehat{A}$ exists;

{\rm (b)}.  $\left(I_m-A_0A_0^{\dag}\right)A_1\left(I_n-A_0^{\dag}A_0\right)=0$;

{\rm (c)}.  the DMPGI of $\widehat{A}$ exists.
\end{theorem}

\begin{proof}
(a)$\Rightarrow$(b):
If the dual $r$-rank decomposition of $\widehat{A}$ exists,
according to Theorem \ref{2.2},
we can get (\ref{2.15}).
It follows from (\ref{20220216-7}) that $\left(I_m-A_0A_0^{\dag}\right)A_1\left(I_n-A_0^{\dag}A_0\right)=0$ holds.

(b)$\Leftarrow$(a):
When $\left(I_m-A_0A_0^{\dag}\right)A_1\left(I_n-A_0^{\dag}A_0\right)=0$ holds,
by applying (\ref{20220216-7})
we   get $\left(I_{m}-A_2A_2^{\dag}\right)A_1\left(I_{n}-A_4^{\dag}A_4\right)=0$.
It follows from  Theorem \ref{2.2},
  that the dual $r$-rank decomposition of $\widehat{A}$ exists.

 Since  DMPGI of $\widehat{A}$ exists if and only if   $\left(I_m-A_0A_0^{\dag}\right)A_1\left(I_n-A_0^{\dag}A_0\right)=0$,
then (b)$\Leftrightarrow$(c).
\end{proof}

\begin{theorem}
\label{2.6}
Let
 $\widehat{A}\in\mathbb{D}_{m\times n}$,
 $\widehat{A}= A_0+\epsilon A_1$,
${\rm rk}(A_0)=r$,
 the dual $r$-rank decomposition of $\widehat{A}$ exist,
and
the dual $r$-rank decomposition of $\widehat{A}$ be $\widehat{A}=\widehat{A_1}\widehat{A_2}$.
Then
\begin{align}
\label{2.17}
\widehat{A}^{\dag}&=\widehat{A_2}^{\dag}\widehat{A_1}^{\dag}\\
\label{2.17-1}
&=\widehat{A_2}^{T}\left(\widehat{A_2}\widehat{A_2}^{T}\right)^{-1}
\left(\widehat{A_1}^{T}\widehat{A_1}\right)^{-1}\widehat{A_1}^{T}.
\end{align}
\end{theorem}

\begin{proof}
Since the dual $r$-rank decomposition of $\widehat{A}$ exists,
from Theorem {\ref{2.4}}, we see that the DMPGI of $\widehat{A}$ exists. 
Let $\widehat{A}=\widehat{A_1}\widehat{A_2}$ be a dual $r$-rank decomposition of $\widehat{A}$,
and denote
$$\widehat{X}
=
\widehat{A_2}^{T}\left(\widehat{A_2}\widehat{A_2}^{T}\right)^{-1}
\left(\widehat{A_1}^{T}\widehat{A_1}\right)^{-1}\widehat{A_1}^{T}.$$
We verify that $\widehat{X}$ satisfies the four dual Penrose equations(\ref{20220216-8}):
\begin{align}
\nonumber
&(1)\ \widehat{A}\widehat{X}\widehat{A}=
\widehat{A_1}\widehat{A_2}\widehat{A_2}^{T}\left(\widehat{A_2}\widehat{A_2}^{T}\right)^{-1}\left(\widehat{A_1}^{T}\widehat{A_1}\right)^{-1}\widehat{A_1}^{T}\widehat{A_1}\widehat{A_2}
=\widehat{A};
\\
\nonumber
&(2)\ \widehat{X}\widehat{A}\widehat{X}
=\widehat{A_2}^{T}\left(\widehat{A_2}\widehat{A_2}^{T}\right)^{-1}\left(\widehat{A_1}^{T}\widehat{A_1}\right)^{-1}\widehat{A_1}^{T}\widehat{A_1}\widehat{A_2}
\widehat{A_2}^{T}\left(\widehat{A_2}\widehat{A_2}^{T}\right)^{-1}\left(\widehat{A_1}^{T}\widehat{A_1}\right)^{-1}\widehat{A_1}^{T}
=\widehat{X};
\\
\nonumber
&(3)\ \left(\widehat{A}\widehat{X}\right)^{T}=
\left(\widehat{A_1}\widehat{A_2}\widehat{A_2}^{T}\left(\widehat{A_2}\widehat{A_2}^{T}\right)^{-1}\left(\widehat{A_1}^{T}\widehat{A_1}\right)^{-1}\widehat{A_1}^{T}\right)^{T}
=\widehat{A_1}\left(\widehat{A_1}^{T}\widehat{A_1}\right)^{-1}\widehat{A_1}^{T}
=\widehat{A}\widehat{X};
\\
\nonumber
&(4)\ \left(\widehat{X}\widehat{A}\right)^{T}=
\left(\widehat{A_2}^{T}\left(\widehat{A_2}\widehat{A_2}^{T}\right)^{-1}\left(\widehat{A_1}^{T}\widehat{A_1}\right)^{-1}\widehat{A_1}^{T}\widehat{A_1}\widehat{A_2}\right)^{T}
=\widehat{A_2}^{T}\left(\widehat{A_2}\widehat{A_2}^{T}\right)^{-1}\widehat{A_2}
=\widehat{X}\widehat{A}.
\end{align}
Since $\widehat{A}^{\dag}$ satisfying  the four equations is unique,
then $\widehat{X}=\widehat{A}^{\dag}$.

Furthermore,
according to Lemma \ref{2.5},
we see
 $\widehat{A_1}^{\dag}=\left(\widehat{A_1}^{T}\widehat{A_1}\right)^{-1}\widehat{A_1}^{T}$
 and
$\widehat{A_2}^{\dag}=\widehat{A_2}^{T}\left(\widehat{A_2}\widehat{A_2}^{T}\right)^{-1}$.
So, $\widehat{A}^{\dag}$ can be further expressed as
$\widehat{A}^{\dag}=\widehat{A_2}^{\dag}\widehat{A_1}^{\dag}$,
that is,  (\ref{2.17}).
\end{proof}

\begin{theorem}
\label{2.7}
Let
 $\widehat{A}\in\mathbb{D}_{m\times n}$,
 $\widehat{A}= A_0+\epsilon A_1$ and
${\rm rk}(A_0)=r$.
 Let $A_0=A_2A_4$ be a full rank decomposition of $A_0$.
Let
 $\widehat{A}=\widehat{A_1}\widehat{A_2}$ be a dual $r$-rank decomposition of $\widehat{A}$
where $\widehat{A_1}=A_2+\epsilon A_3$ and $\widehat{A_2}=A_4+\epsilon A_5$.
Then the DMPGI of $\widehat{A}$ exists,
and
\begin{align}
\nonumber
\widehat{A}^{\dag}&=A_4^{\dag}A_2^{\dag}+
\epsilon\left(A_4^{\dag}\left(A_2^{T}A_2\right)^{-1}
\left(A_3^{T}
-
 Q_{A_2,A_3}^{S}A_2^{\dag}\right)\right.
 \\
\label{2.18}
 &
 \qquad \qquad \qquad \qquad   \qquad \qquad
\left.+\left(A_5^{T}
-A_4^{\dag}Q_{A_4^{T},A_5^{T}}^{S}\right)\left(A_4A_4^{T}\right)^{-1}A_2^{\dag}\right),
\end{align}
where $Q_{A_4^{T},A_5^{T}}^{S}=A_4A_5^{T}+A_5A_4^{T}$ and $Q_{A_2,A_3}^{S}=A_2^{T}A_3+A_3^{T}A_2$.
\end{theorem}
\begin{proof}
According to Lemma \ref{2.5},
by substituting  (\ref{2.20}) and (\ref{2.19}) into  (\ref{2.17}),
we can get
\begin{align}
\nonumber
\widehat{A}^{\dag}&=\left(A_4^{T}\left(A_4A_4^{T}\right)^{-1}
+
\epsilon\left(A_5^{T}\left(A_4A_4^{T}\right)^{-1}
-A_4^{T}\left(A_4A_4^{T}\right)^{-1}Q_{A_4^{T},A_5^{T}}^{S}
\left(A_4A_4^{T}\right)^{-1}\right)\right)
\\&
\nonumber
\quad
\left(\left(A_2^{T}A_2\right)^{-1}A_2^{T}
+
\epsilon\left(\left(A_2^{T}A_2\right)^{-1}A_3^{T}-\left(A_2^{T}A_2\right)^{-1}
Q_{A_2,A_3}^{S}\left(A_2^{T}A_2\right)^{-1}A_2^{T}\right)\right).
\end{align}
Furthermore,
from %
$A_4^{\dag}=A_4^{T}\left(A_4A_4^{T}\right)^{-1}$ and $A_2^{\dag}=\left(A_2^{T}A_2\right)^{-1}A_2^{T}$,
we can get the formula for DMPGI $\widehat{A}^\dag$ as shown in (\ref{2.18}).
\end{proof}

Based on Theorem \ref{2.7},
the detailed calculation process of DMPGI is given  below,
and one corresponding example is also given to verify.

\textbf{(1).} Input matrix $A_0,A_1$,
and the form of the dual matrix $\widehat{A}$ is
$\widehat{A}=A_0+\epsilon A_1,\ A_i \in \mathbb{R}_{m\times n},\ {\rm rk}(A_0)=r$;

\textbf{(2).} According to the method of calculating dual $r$-rank decomposition,
we get
 $\widehat{A}=\widehat{A_1}\widehat{A_2}$
 where
$\widehat{A_1}=A_2+\epsilon A_3$ is an $r$-column full rank  dual matrix and
$\widehat{A_2}=A_4+\epsilon A_5$ is an $r$-row full rank  dual matrix;

\textbf{(3).}  Calculate
$ A_4^{\dag}$, $A_2^{\dag}$ and  $ A_4^{\dag}A_2^{\dag}$;

\textbf{(4).} Calculate
$ A_4^{\dag}\left(A_2^{T}A_2\right)^{-1}
\left(A_3^{T}
-
 Q_{A_2,A_3}^{S}A_2^{\dag}\right)
+\left(A_5^{T}
-A_4^{\dag}Q_{A_4^{T},A_5^{T}}^{S}\right)
\left(A_4A_4^{T}\right)^{-1}A_2^{\dag} $;

\textbf{(5).} Get the DMPGI $\widehat{A}^{\dag}$ of $\widehat{A}$.

\begin{example}
Let $\widehat{A}$, $A_2$, $A_3$, $A_4$ and $A_5$ be as given in Example {\rm \ref{2.3}}.
By applying     (\ref{2.18}),
we can get the following result: \\
$\widehat{X}
=
A_4^{\dag}A_2^{\dag}+\epsilon \left(A_4^{\dag}\left(A_2^{T}A_2\right)^{-1}
\left(A_3^{T}
-
 Q_{A_2,A_3}^{S}A_2^{\dag}\right)
+\left(A_5^{T}
-A_4^{\dag}Q_{A_4^{T},A_5^{T}}^{S}\right)\left(A_4A_4^{T}\right)^{-1}A_2^{\dag}\right)$
\\
{\footnotesize$=\left[\begin{matrix}
1&   0&   \frac{1}{3}\\
0&   1&   \frac{1}{3}
\end{matrix}\right]^{T}
\left(\left[\begin{matrix}
1&   0&   \frac{1}{3}\\
0&   1&   \frac{1}{3}
\end{matrix}\right]
\left[\begin{matrix}
1&   0&   \frac{1}{3}\\
0&   1&   \frac{1}{3}
\end{matrix}\right]^{T}\right)^{-1}
\left(\left[\begin{matrix}
    1&   2\\
    2&   1\\
    3&   3
\end{matrix}\right]^{T}
\left[\begin{matrix}
    1&   2\\
    2&   1\\
    3&   3
\end{matrix}\right]\right)^{-1}
\left[\begin{matrix}
    1&   2\\
    2&   1\\
    3&   3
\end{matrix}\right]^{T}
\\
+\epsilon\left\{\left[\begin{matrix}
1&   0&   \frac{1}{3}\\
0&   1&   \frac{1}{3}
\end{matrix}\right]^{T}
\left(\left[\begin{matrix}
1&   0&   \frac{1}{3}\\
0&   1&   \frac{1}{3}
\end{matrix}\right]
\left[\begin{matrix}
1&   0&   \frac{1}{3}\\
0&   1&   \frac{1}{3}
\end{matrix}\right]^{T}\right)^{-1}
\left(\left[\begin{matrix}
    1&   2\\
    2&   1\\
    3&   3
\end{matrix}\right]^{T}
\left[\begin{matrix}
    1&   2\\
    2&   1\\
    3&   3
\end{matrix}\right]\right)^{-1}
\left[\begin{matrix}
\frac{1}{2}&     -\frac{1}{2}\\
      0&         -\frac{1}{2}\\
 \frac{3}{2}&          -4
\end{matrix}\right]^{T}\right.\\
-\left[\begin{matrix}
1&   0&   \frac{1}{3}\\
0&   1&   \frac{1}{3}
\end{matrix}\right]^{T}
\left(\left[\begin{matrix}
1&   0&   \frac{1}{3}\\
0&   1&   \frac{1}{3}
\end{matrix}\right]
\left[\begin{matrix}
1&   0&   \frac{1}{3}\\
0&   1&   \frac{1}{3}
\end{matrix}\right]^{T}\right)^{-1}
\left(\left[\begin{matrix}
    1&   2\\
    2&   1\\
    3&   3
\end{matrix}\right]^{T}
\left[\begin{matrix}
\frac{1}{2}&     -\frac{1}{2}\\
      0&         -\frac{1}{2}\\
 \frac{3}{2}&          -4
\end{matrix}\right]+
\left[\begin{matrix}
\frac{1}{2}&     -\frac{1}{2}\\
      0&         -\frac{1}{2}\\
 \frac{3}{2}&          -4
\end{matrix}\right]^{T}
\left[\begin{matrix}
    1&   2\\
    2&   1\\
    3&   3
\end{matrix}\right]\right)
\\ \times
\left(\left[\begin{matrix}
    1&   2\\
    2&   1\\
    3&   3
\end{matrix}\right]^{T}
\left[\begin{matrix}
    1&   2\\
    2&   1\\
    3&   3
\end{matrix}\right]\right)^{-1}
\left[\begin{matrix}
    1&   2\\
    2&   1\\
    3&   3
\end{matrix}\right]^{T}
+\left[\begin{matrix}
\frac{3}{2}&   \frac{13}{6}&    \frac{29}{9}\\
    -1&        \frac{7}{6}&     \frac{31}{18}
\end{matrix}\right]^{T}
\left(\left[\begin{matrix}
1&   0&   \frac{1}{3}\\
0&   1&   \frac{1}{3}
\end{matrix}\right]
\left[\begin{matrix}
1&   0&   \frac{1}{3}\\
0&   1&   \frac{1}{3}
\end{matrix}\right]^{T}\right)^{-1}\\
\times
\left(\left[\begin{matrix}
    1&   2\\
    2&   1\\
    3&   3
\end{matrix}\right]^{T}
\left[\begin{matrix}
    1&   2\\
    2&   1\\
    3&   3
\end{matrix}\right]\right)^{-1}
\left[\begin{matrix}
    1&   2\\
    2&   1\\
    3&   3
\end{matrix}\right]^{T}
+\left[\begin{matrix}
1&   0&   \frac{1}{3}\\
0&   1&   \frac{1}{3}
\end{matrix}\right]^{T}
\left(\left[\begin{matrix}
1&   0&   \frac{1}{3}\\
0&   1&   \frac{1}{3}
\end{matrix}\right]
\left[\begin{matrix}
1&   0&   \frac{1}{3}\\
0&   1&   \frac{1}{3}
\end{matrix}\right]^{T}\right)^{-1}
\\ \times
\left(\left[\begin{matrix}
1&   0&   \frac{1}{3}\\
0&   1&   \frac{1}{3}
\end{matrix}\right]
\left[\begin{matrix}
\frac{3}{2}&   \frac{13}{6}&    \frac{29}{9}\\
    -1&        \frac{7}{6}&     \frac{31}{18}
\end{matrix}\right]^{T}+\\
\left[\begin{matrix}
\frac{3}{2}&   \frac{13}{6}&    \frac{29}{9}\\
    -1&        \frac{7}{6}&     \frac{31}{18}
\end{matrix}\right]
\left[\begin{matrix}
1&   0&   \frac{1}{3}\\
0&   1&   \frac{1}{3}
\end{matrix}\right]^{T}\right)
\\ \times
\left.
\left(\left[\begin{matrix}
1&   0&   \frac{1}{3}\\
0&   1&   \frac{1}{3}
\end{matrix}\right]
\left[\begin{matrix}
1&   0&   \frac{1}{3}\\
0&   1&   \frac{1}{3}
\end{matrix}\right]^{T}\right)^{-1}
\left(\left[\begin{matrix}
    1&   2\\
    2&   1\\
    3&   3
\end{matrix}\right]^{T}
\left[\begin{matrix}
    1&   2\\
    2&   1\\
    3&   3
\end{matrix}\right]\right)^{-1}
\left[\begin{matrix}
    1&   2\\
    2&   1\\
    3&   3
\end{matrix}\right]^{T}\right\}
\\
=\left[\begin{matrix}
    -\frac{5}{11}&    \frac{6}{11}&  \frac{1}{11}\\
     \frac{6}{11}&   -\frac{5}{11}&  \frac{1}{11}\\
     \frac{1}{33}&    \frac{1}{33}&  \frac{2}{33}
\end{matrix}\right]+\epsilon
\left[\begin{matrix}
    -\frac{31}{33}&    -\frac{16}{33}&     \frac{1}{33}\\
     \frac{2}{11}&      \frac{7}{11}&     -\frac{8}{11}\\
    -\frac{25}{99}&     \frac{38}{99}&     \frac{10}{99}
\end{matrix}\right].$}

Furthermore,
 we prove that $\widehat{X}$ satisfies the following four Penrose equations::

{\rm (1)}. \  {\footnotesize$
\widehat{A}\widehat{X}\widehat{A}
=\left[\begin{matrix}
    1&   2&  1\\
    2&   1&  1\\
    3&   3&  2
\end{matrix}\right]+\epsilon
\left[\begin{matrix}
    1&   4&  7\\
    2&   5&  8\\
    3&   6&  14
\end{matrix}\right]
=\widehat{A};
$}

{\rm (2)}. \   {\footnotesize$
\widehat{X}\widehat{A}\widehat{X}
=\left[\begin{matrix}
    -\frac{5}{11}&    \frac{6}{11}&  \frac{1}{11}\\
     \frac{6}{11}&   -\frac{5}{11}&  \frac{1}{11}\\
     \frac{1}{33}&    \frac{1}{33}&  \frac{2}{33}
\end{matrix}\right]+\epsilon
\left[\begin{matrix}
    -\frac{31}{33}&    -\frac{16}{33}&     \frac{1}{33}\\
     \frac{2}{11}&      \frac{7}{11}&     -\frac{8}{11}\\
    -\frac{25}{99}&     \frac{38}{99}&     \frac{10}{99}
\end{matrix}\right]
=\widehat{X};
$}

{\rm (3)}. \  {\footnotesize$
\left(\widehat{A}\widehat{X}\right)^{T}
=\left(\left[\begin{matrix}
    \frac{2}{3}&   -\frac{1}{3}&  \frac{1}{3}\\
   -\frac{1}{3}&    \frac{2}{3}&  \frac{1}{3}\\
    \frac{1}{3}&    \frac{1}{3}&  \frac{2}{3}
\end{matrix}\right]+\epsilon\left[\begin{matrix}
     \frac{10}{9}&   \frac{1}{9}&  -\frac{4}{9}\\
     \frac{1}{9}&   -\frac{8}{9}&   \frac{5}{9}\\
    -\frac{4}{9}&    \frac{5}{9}&  -\frac{2}{9}
\end{matrix}\right]\right)^{T}
=\widehat{A}\widehat{X};
$}

{\rm (4)}. \ {\footnotesize
$
\left(\widehat{X}\widehat{A}\right)^{T}
 =\left(\left[\begin{matrix}
    \frac{10}{11}&   -\frac{1}{11}&   \frac{3}{11}\\
   -\frac{1}{11}&     \frac{10}{11}&  \frac{3}{11}\\
    \frac{3}{11}&     \frac{3}{11}&   \frac{2}{11}
\end{matrix}\right]+\epsilon\left[\begin{matrix}
    -\frac{10}{11}&   -\frac{9}{11}&   \frac{12}{11}\\
    -\frac{9}{11}&    -\frac{8}{11}&   \frac{9}{11}\\
     \frac{12}{11}&    \frac{9}{11}&   \frac{18}{11}
\end{matrix}\right]\right)^{T}
=\widehat{X}\widehat{A}.
$}
\\
Therefore,
$
\widehat{A}^\dag=\widehat{X}=\left[\begin{matrix}
    -\frac{5}{11}&    \frac{6}{11}&  \frac{1}{11}\\
     \frac{6}{11}&   -\frac{5}{11}&  \frac{1}{11}\\
     \frac{1}{33}&    \frac{1}{33}&  \frac{2}{33}
\end{matrix}\right]+\epsilon
\left[\begin{matrix}
    -\frac{31}{33}&    -\frac{16}{33}&     \frac{1}{33}\\
     \frac{2}{11}&      \frac{7}{11}&     -\frac{8}{11}\\
    -\frac{25}{99}&     \frac{38}{99}&     \frac{10}{99}
\end{matrix}\right]$.
\end{example}

\subsection{Dual Idempotent Matrix}
In {\cite{Udwadia2021}},
Udwadia discussed  several types of special dual idempotent matrices,
such as $\widehat{A}\widehat{A}^{\dag}$,
$\widehat{A}^{\dag}\widehat{A}$,
$ I_m-\widehat{A}\widehat{A}^{\dag}$
and $I_n-\widehat{A}^{\dag}\widehat{A}$.
In this subsection,
  we
give some  characterizations of dual idempotent matrix  and its DMPGI  by applying the dual $r$-rank decomposition.

\begin{definition}[{\cite{Udwadia2021}}]
\label{3.1-1}
Let
 $\widehat{A}\in\mathbb{D}_{n\times n}$,
 $\widehat{A}= A_0+\epsilon A_1$
  and
  ${\rm rk}(A_0)=r$.
If $\widehat{A}$ satisfies $\widehat{A}^{2}=\widehat{A}$,
then $\widehat{A}$ is called dual idempotent matrix.
\end{definition}

\begin{theorem}
\label{3.1}
Let
 $\widehat{A}\in\mathbb{D}_{n\times n}$,
 $\widehat{A}= A_0+\epsilon A_1$
 and ${\rm rk}(A_0)=r$.
 Then
$\widehat{A}$ is a dual idempotent matrix if and only if
\begin{align}
\label{3.11}
A_0  =A_0^{2} \mbox{ \ \ and \ }
 A_1 = A_0A_1+A_1A_0 .
\end{align}
\end{theorem}

\begin{proof}
$"\Rightarrow"$:
If $\widehat{A}=A_0+\epsilon A_1$ is a dual idempotent matrix,
then we have
 $\widehat{A}^{2}=\widehat{A}$ and
$A_0^{2}+\epsilon\left(A_0A_1+A_1A_0\right)=A_0+\epsilon A_1$.
Therefore (\ref{3.11}) is established.

$"\Leftarrow"$:
Since $\widehat{A}=A_0+\epsilon A_1$,
 it is obvious that $\widehat{A}^{2}=A_0^{2}+\epsilon\left(A_0A_1+A_1A_0\right)$.
It follows from (\ref{3.11}) that
$\widehat{A}^{2}=A_0 +\epsilon A_1 =\widehat{A}$.
Therefore, according to Definition \ref{3.1-1},
we see that $\widehat{A}$ is a dual idempotent matrix.
\end{proof}

\begin{corollary}
Let
 $\widehat{A}\in\mathbb{D}_{n\times n}$,
 $\widehat{A}= A_0+\epsilon A_1$
 and ${\rm rk}(A_0)=r$.
If $\widehat{A}$ is a dual idempotent matrix,
and the real part matrix $A_0$ is invertible,
then $\widehat{A}=I_n$.
\end{corollary}

\begin{proof}
According to the Theorem \ref{3.1},
if $\widehat{A}$ is a dual idempotent matrix,
then the equation (\ref{3.11}) holds.
If the real matrix $A_0$ is invertible,
we can get $A_0=I_n$.
Since $A_0=I_n$ and $A_1 = A_0A_1+A_1A_0$,
it is easy to check that $A_1=0$.
Hence,  $\widehat{A}=I_n$.
\end{proof}

\begin{theorem}
\label{3.0}
Let
 $\widehat{A}\in\mathbb{D}_{n\times n}$,
 $\widehat{A}= A_0+\epsilon A_1$ and
${\rm rk}(A_0)=r$.
Let $A_0=A_2A_4$ be a full rank decomposition of $A_0$.
Then the dual $r$-rank decomposition of $\widehat{A}$ exists,
and
 \begin{align}
 \widehat{A}=
 \left(A_2+\epsilon A_1A_2\right)\left(A_4+\epsilon A_4A_1\right)
 \end{align}
 which is a   dual $r$-rank decomposition of $\widehat{A}$.
\end{theorem}

\begin{proof}
Let $\widehat{A}$ be a dual idempotent matrix,
then the equation (\ref{3.11}) holds.
Let
$A_0=A_2A_4$ be a full rank decomposition of $A_0$,
where $A_2$ is a column full  rank matrix,
and $A_4$ is a row full rank matrix.
Write
$\widehat{X}=A_2+\epsilon A_1A_2  $
and
$\widehat{Y}=A_4+\epsilon A_4A_1 $.
It is obvious  that
$\widehat{X}$ is an $r$-column full rank dual  matrix
and
 $\widehat{Y}$ is an $r$-row full rank dual  matrix.
 It follows from   (\ref{3.11})
 that
\begin{align*}
\widehat{X} \widehat{Y}
&=
\left(A_2+\epsilon A_1A_2\right)\left(A_4+\epsilon A_4A_1\right)
=A_2A_4+\epsilon \left(A_2A_4A_1+\epsilon A_1A_2A_4\right)
\\
&=A_0+\epsilon \left(A_0A_1+ A_1A_0\right)
=A_0+\epsilon A_1.
\end{align*}
Therefore, the dual $r$-rank decomposition of $\widehat{A}$ exists and
$\widehat{A}=
 \left(A_2+\epsilon A_1A_2\right)\left(A_4+\epsilon A_4A_1\right)$
 is a dual $r$-rank decomposition of $\widehat{A}$.
\end{proof}

\begin{theorem}
Let
 $\widehat{A}= A_0+\epsilon A_1\in\mathbb{D}_{n\times n}$ be a dual idempotent matrix.
Then
\begin{align}
\label{3.112}
\widehat{A}^{\dag}
&
=
A_0^{\dag}+\epsilon\left(A_0^{\dag}A_1^{T}+A_1^{T}A_0^{\dag}-A_0^{\dag}\left(A_1 +A_1^{T}\right)A_0 A_0^{\dag}-A_0^{\dag}
A_0\left( A_1^{T}+A_1\right)A_0^{\dag}\right).
\end{align}
\end{theorem}
\begin{proof}
If $\widehat{A}$ is a dual idempotent matrix,
according to Theorem \ref{3.0},
the dual $r$-rank decomposition of $\widehat{A}$ exists.
Let $A_0=A_2A_4$ be a full rank decomposition of $A_0$
and
$\widehat{A}=\widehat{A_1}\widehat{A_2}$
be a dual $r$-rank decomposition of $\widehat{A}$
where
$\widehat{A_1}=A_2+\epsilon A_1A_2$
and
$\widehat{A_2}=A_4+\epsilon A_4A_1$.
Because $\widehat{A_1}$ is an $r$-column full rank  dual matrix and $\widehat{A_2}$ is an $r$-row full rank dual  matrix,
then
\begin{align}
\nonumber
\left\{
\begin{aligned}
\left(\widehat{A_1}^{T}\widehat{A_1}\right)^{-1}
&
=
(A_2^{T}A_2)^{-1}-\epsilon \left(A_2^{\dag}(A_1^{T}+A_1)\left(A_2^{\dag}\right)^{T}\right) \\
\nonumber
\left(\widehat{A_2}\widehat{A_2}^{T}\right)^{-1}
&
=
(A_4A_4^{T})^{-1}-\epsilon \left(\left(A_4^{\dag}\right)^{T}(A_1^{T}+A_1)A_4^{\dag}\right)
\end{aligned}\right..
\end{align}
By applying   (\ref{2.20-1}),  (\ref{2.20-2}), (\ref{20220216-7})
and the above equations  to  (\ref{2.17-1})
\begin{align}
\nonumber
\widehat{A}^{\dag}
&=
A_0^{\dag}+\epsilon\left(A_0^{\dag}A_1^{T}-A_0^{\dag}\left(A_1A_2+A_1^{T}A_2\right)A_2^{\dag}+A_1^{T}A_0^{\dag}-A_4^{\dag}
\left(A_4A_1^{T}+A_4A_1\right)A_0^{\dag}\right)
\\
\nonumber
&=
A_0^{\dag}+\epsilon\left(A_0^{\dag}A_1^{T}-A_0^{\dag}\left(A_1 +A_1^{T}\right)A_2 A_2^{\dag}+A_1^{T}A_0^{\dag}-A_4^{\dag}
A_4\left( A_1^{T}+A_1\right)A_0^{\dag}\right)
\\
\nonumber
&=
A_0^{\dag}+\epsilon\left(A_0^{\dag}A_1^{T}+A_1^{T}A_0^{\dag}-A_0^{\dag}\left(A_1 +A_1^{T}\right)A_0 A_0^{\dag}-A_0^{\dag}
A_0\left( A_1^{T}+A_1\right)A_0^{\dag}\right).
\end{align}
Therefore,
we get (\ref{3.112}).
\end{proof}

\begin{theorem}
\label{3.2}
Let
 $\widehat{A}\in\mathbb{D}_{n\times n}$,
 $\widehat{A}= A_0+\epsilon A_1$ and
${\rm rk}(A_0)=r$.
Let $\widehat{A}=\widehat{A_1}\widehat{A_2}$
       be a dual r-rank decomposition of $\widehat{A}$.
Then   $\widehat{A}$ is a dual idempotent matrix
if and only if
$\widehat{A_2}\widehat{A_1}=I_r$.
\end{theorem}
\begin{proof}
$"\Rightarrow"$:
Let $\widehat{A}$ be  a dual idempotent matrix,
then the dual $r$-rank decomposition of $\widehat{A}$ exists.
Let $A_0=A_2A_4$ be a full rank decomposition of $A_0$,
and
$\widehat{A}=\widehat{A_1}\widehat{A_2}=(A_2+\epsilon Y)(A_4+\epsilon X)$
be a dual $r$-rank decomposition of $\widehat{A}$.
Since
$\widehat{A}$ is a  dual idempotent matrix,
 by the first equation in   (\ref{3.11}),
 we see that   $A_0$ is an idempotent matrix,
and $A_4A_2=I_r$.
Therefore,
\begin{align}
\label{20220216-9}
\widehat{A_2}\widehat{A_1}=I_r+\epsilon Z.
\end{align}

Because $\widehat{A}$ is a dual idempotent matrix,
we have
$\widehat{A_1}\widehat{A_2}\widehat{A_1}\widehat{A_2}
=\widehat{A_1}\widehat{A_2}$,
\begin{align}
\nonumber
\widehat{A_1}\widehat{A_2}
=
(A_2+\epsilon Y)(A_4+\epsilon X)
=
A_2A_4+\epsilon(A_2X+YA_4)
\end{align}
and
\begin{align}
\nonumber
\widehat{A_1}\widehat{A_2}\widehat{A_1}\widehat{A_2}
=
(A_2+\epsilon Y)(I_r+\epsilon Z)(A_4+\epsilon X)
=A_2A_4+\epsilon(A_2X+A_2ZA_4+YA_4).
\end{align}
Therefore,
$A_2ZA_4=0$.
Since
$A_2$ is a column full rank matrix
and
$A_4$ is a row full  rank matrix,
$Z=0$.
It follows   from (\ref{20220216-9}) that  $\widehat{A_2}\widehat{A_1}=I_r$.

$"\Leftarrow"$:
Let $\widehat{A_2}\widehat{A_1}=I_r$.
Then
$\widehat{A}^{2}=\widehat{A_1}\widehat{A_2}\widehat{A_1}\widehat{A_2}
=\widehat{A_1}I_r\widehat{A_2}=\widehat{A_1}\widehat{A_2}=\widehat{A}$,
that is, $\widehat{A}$ is a dual idempotent matrix.
\end{proof}

\subsection{Dual EP Matrix}
This subsection introduces one special dual matrix: dual EP matrix,
and
considers characterizations, dual $r$-rank decomposition and DMPGI of the special matrix.

\begin{definition}
\label{3.3}
Let  $\widehat{A}\in\mathbb{D}_{n\times n}$,
and
$\widehat{A}^{\dag}$  exist.
If
\begin{align}
\label{2.21}
\widehat{A}\widehat{A}^{\dag}=\widehat{A}^{\dag}\widehat{A},
\end{align}
then $\widehat{A}$ is called a dual EP matrix.
\end{definition}

\begin{theorem}
\label{3.4}
Let
 $\widehat{A}\in\mathbb{D}_{n\times n}$,
 $\widehat{A}= A_0+\epsilon A_1$ and
${\rm rk}(A_0)=r$.
  Let $\widehat{A}=\widehat{A_1}\widehat{A_2}$ be a dual $r$-rank decomposition of $\widehat{A}$.
Then $\widehat{A}$ is a dual EP matrix if and only if
\begin{align}
\label{2.22}
\widehat{A_1}\widehat{A_1}^{\dag}=\widehat{A_2}^{\dag}\widehat{A_2}.
\end{align}
\end{theorem}

\begin{proof}
$"\Rightarrow"$:
Since the dual $r$-rank decomposition of $\widehat{A}$ exists,
  the DMPGI of $\widehat{A}$ exists.
Let
$\widehat{A}=\widehat{A_1}\widehat{A_2}$ be the dual $r$-rank decomposition of $\widehat{A}$,
and
$\widehat{A}$ be a dual EP matrix.
According to Definition \ref{3.3},
we can get the equation ({\ref{2.21}}).
Then by applying (\ref{2.17-1})
 to   ({\ref{2.21}}),
we  get
\begin{align}
\nonumber
\widehat{A_1}\widehat{A_2}\widehat{A_2}^{T}\left(\widehat{A_2}\widehat{A_2}^{T}\right)^{-1}
\left(\widehat{A_1}^{T}\widehat{A_1}\right)^{-1}\widehat{A_1}^{T}
&=\widehat{A_2}^{T}\left(\widehat{A_2}\widehat{A_2}^{T}\right)^{-1}\left(\widehat{A_1}^{T}
\widehat{A_1}\right)^{-1}\widehat{A_1}^{T}\widehat{A_1}\widehat{A_2},
\end{align}
that is,
$\widehat{A_1}\left(\widehat{A_1}^{T}\widehat{A_1}\right)^{-1}\widehat{A_1}^{T}
=\widehat{A_2}^{T}\left(\widehat{A_2}\widehat{A_2}^{T}\right)^{-1}\widehat{A_2}$.
It follows from (\ref{2.20-1}) and  (\ref{2.20-2}) that
 we obtain  (\ref{2.22}).

$"\Leftarrow"$:
Conversely,
with the precondition that $\widehat{A_1}$ is an $r$-column full rank  dual matrix and $\widehat{A_2}$ is an $r$-row full rank  dual matrix,
if the equation (\ref{2.22}) holds,
according to Lemma \ref{2.5},
we have
$\widehat{A_1}^{\dag}=\left(\widehat{A_1}^{T}\widehat{A_1}\right)^{-1}\widehat{A_1}^{T}$
and
$\widehat{A_2}^{\dag}
=\widehat{A_2}^{T}\left(\widehat{A_2}\widehat{A_2}^{T}\right)^{-1}$.
Then applying these two equations to the equation (\ref{2.22}),
we  get
$\widehat{A_1}\left(\widehat{A_1}^{T}\widehat{A_1}\right)^{-1}\widehat{A_1}^{T}
=\widehat{A_2}^{T}\left(\widehat{A_2}\widehat{A_2}^{T}\right)^{-1}
\widehat{A_2}$.
Therefore,
 $$\widehat{A_1}\widehat{A_2}\widehat{A_2}^{T}\left(\widehat{A_2}\widehat{A_2}^{T}\right)^{-1}
 \left(\widehat{A_1}^{T}\widehat{A_1}\right)^{-1}\widehat{A_1}^{T}
=\widehat{A_2}^{T}
\left(\widehat{A_2}\widehat{A_2}^{T}\right)^{-1}\left(\widehat{A_1}^{T}
\widehat{A_1}\right)^{-1}\widehat{A_1}^{T}\widehat{A_1}\widehat{A_2}.$$
Hence,
the equation ({\ref{2.21}}) holds, that is, $\widehat{A}$ is a dual EP matrix.
\end{proof}

\begin{theorem}
\label{3.5-0}
Let
 $\widehat{A}\in\mathbb{D}_{n\times n}$,
 $\widehat{A}= A_0+\epsilon A_1$,
and the DMPGI of $\widehat{A}$ exist.
Then $\widehat{A}$ is a dual EP matrix if and only if
\begin{subnumcases}{}
\label{3.5-0-1}
 A_0 A_0^\dag =A_0^\dag A_0,
\\
\label{3.5-0-2}
 \left(I_n-A_0^\dag A_0\right)A_1A_0^\dag
=\left(A_0^\dag A_1\left(I_n-A_0^\dag A_0\right)\right)^T.
 \end{subnumcases}
\end{theorem}

\begin{proof}
By applying (\ref{The-DMPGI}) and Definition \ref{3.3},
  we can get that $\widehat{A}$ is a dual EP matrix if and only if
\begin{align}
\label{20220210-1}
\left(A_0+\epsilon A_1\right)   \left(A_0^\dag- \epsilon R\right)
=
\left(A_0^\dag- \epsilon R\right)   \left(A_0+\epsilon A_1\right),
\end{align}
in which $R=A_0^{\dag}A_1A_0^{\dag}
 -\left(A_0^{T}A_0\right)^{\dag}A_1^{T}\left(I_n-A_0A_0^{\dag}\right)
 -\left(I_n-A_0^{\dag}A_0\right)A_1^{T}\left(A_0A_0^{T}\right)^{\dag}$.

$"\Rightarrow"$:
Let $\widehat{A}$ be a dual EP matrix.
By applying (\ref{20220210-1}),
we see that
 \begin{align}
\label{20220210-2}
A_0A_0^\dag+\epsilon  \left( A_1A_0^\dag-A_0R\right)=
A_0^\dag A_0  + \epsilon  \left( A_0^\dag A_1-RA_0\right).
\end{align}
Therefore,
we get (\ref{3.5-0-1})
and
 \begin{align}
\label{20220210-8}
A_1A_0^\dag-A_0R=A_0^\dag A_1-RA_0.
\end{align}

Since  $A_0A_0^\dag=A_0^\dag A_0$,
  $A_0$ is EP.
Then there exists an orthogonal matrix $U$ such that
 \begin{align}
\label{20220210-3}
A_0=U\left[\begin{matrix} T&0\\ 0&0 \end{matrix}  \right]U^{T},
\end{align}
where $T\in \mathbb{R}_{r\times r}$ is a nonsingular matrix.
It is easy to check that
 \begin{align}
\label{20220210-4}
\left(A_0A_0^{T}\right)^{\dag} A_0 =\left( A_0^{T}\right)^{\dag}.
\end{align}

By applying (\ref{20220210-4}) and  $A_0A_0^\dag=A_0^\dag A_0$,
we see that
 \begin{align}
 \nonumber
A_1A_0^\dag-A_0R
&
=
A_1A_0^\dag- A_0A_0^{\dag}A_1A_0^{\dag}
 +A_0\left(A_0^{T}A_0\right)^{\dag}A_1^{T}
 \left(I_n-A_0A_0^{\dag}\right),
 \\
\label{20220210-5}
 &
 =
\left(I_n- A_0A_0^{\dag}\right)A_1A_0^{\dag}
 + \left(A_0^{T} \right)^{\dag}A_1^{T}\left(I_n-A_0A_0^{\dag}\right),
\end{align}
and
 \begin{align}
 \nonumber
 A_0^\dag A_1-RA_0
&
 =A_0^\dag A_1- A_0^{\dag}A_1A_0^{\dag} A_0
+\left(I_n-A_0^{\dag}A_0\right)A_1^{T}\left(A_0A_0^{T}\right)^{\dag} A_0
\\
\label{20220210-6}
&
 =A_0^\dag A_1\left(I_n- A_0A_0^{\dag} \right)
+\left(I_n-A_0A_0^{\dag}\right)A_1^{T}\left( A_0^{T}\right)^{\dag}.
\end{align}
By substituting (\ref{20220210-5}) and  (\ref{20220210-6})  into
(\ref{20220210-8})
we get
 \begin{align}
 \label{20220210-9}
\left(I_n- A_0A_0^{\dag}\right)\left(A_1A_0^{\dag}-A_1^{T}\left( A_0^{T}\right)^{\dag}\right)
&
 =\left(A_0^\dag A_1 - \left(A_0^{T} \right)^{\dag}
 A_1^{T}\right) \left(I_n- A_0A_0^{\dag} \right).
\end{align}
It is obvious that
$\left(I_n- A_0A_0^{\dag}\right)\left(A_1A_0^{\dag}
-
A_1^{T}\left( A_0^{T}\right)^{\dag}\right)$ is an antisymmetric matrix.

Furthermore, write  \begin{align}
\label{20220210-7}
A_1
=
U\left[\begin{matrix} A_{11}& A_{12}\\  A_{21}& A_{22} \end{matrix}  \right]U^{T},
\end{align}
where $A_{11}\in \mathbb{R}_{r\times r}$.
By applying (\ref{20220210-3}) and (\ref{20220210-7}),
we get
\begin{align}
\nonumber
&\left(I_n- A_0A_0^{\dag}\right)\left(A_1A_0^{\dag}
-
A_1^{T}\left( A_0^{T}\right)^{\dag}\right)
\\
\nonumber
&=
U\left[\begin{matrix} 0&0\\ 0&I_{n-r} \end{matrix}  \right]U^T\left(A_1U\left[\begin{matrix} T^{-1}&0\\ 0&0 \end{matrix}  \right]U^{T}
-
A_1^{T}U\left[\begin{matrix} \left(T^{T}\right)^{-1}&0\\ 0&0 \end{matrix}  \right]U^{T} \right)
\\
\nonumber
&=
U\left[\begin{matrix} 0&0\\ A_{21}T^{-1}-A_{21}^{T} \left(T^{T}\right)^{-1}&0\end{matrix}  \right]U^{T}.
\end{align}
Since
it is an antisymmetric matrix and $A_0A_0^\dag=A_0^\dag A_0$,
it is  obvious  that
\begin{align}
\label{20220210-10}
\left(I_n- A_0A_0^{\dag}\right)\left(A_1A_0^{\dag}-A_1^{T}\left( A_0^{T}\right)^{\dag}\right)=0.
\end{align}
Therefore,  we get (\ref{3.5-0-2}).

$"\Leftarrow"$:  Conversely,
from (\ref{3.5-0-1}),
we get that  $\left(A_0A_0^{T}\right)^{\dag} A_0 =\left( A_0^{T}\right)^{\dag}$,
$A_0$ is EP and  $A_0$ has the decomposition (\ref{20220210-3}).
From (\ref{3.5-0-2}), we have (\ref{20220210-10}).
Therefore, we get (\ref{20220210-9}).

By applying (\ref{3.5-0-1}),
 (\ref{20220210-9})
 and
 $\left(A_0A_0^{T}\right)^{\dag} A_0 =\left( A_0^{T}\right)^{\dag}$,
 we have (\ref{20220210-2}) and (\ref{20220210-8}).
Therefore, we get (\ref{20220210-1}), that is, $\widehat{A}$ is a dual EP matrix.
\end{proof} 

\begin{theorem}
\label{3.5}
Let
 $\widehat{A}\in\mathbb{D}_{n\times n}$,
 $\widehat{A}= A_0+\epsilon A_1$ and
${\rm rk}(A_0)=r$.
Let $A_0=A_2A_4$ be a full rank decomposition of $A_0$.
If the dual $r$-rank decomposition of $\widehat{A}$ exists,
let $\widehat{A}=\widehat{A_1}\widehat{A_2}$ be a dual $r$-rank decomposition of $\widehat{A}$
where $\widehat{A_1}=A_2+\epsilon A_3 \in \mathbb{D}_{n\times r}$ and $\widehat{A_2}=A_4+\epsilon A_5 \in \mathbb{D}_{r\times n}$.
then $\widehat{A}$ is a dual EP matrix if and only if
\begin{subnumcases}{}
\label{2.23}
A_2\left(A_2^{T}A_2\right)^{-1}A_2^{T}=A_4^{T}\left(A_4A_4^{T}\right)^{-1}A_4
\\
\label{2.24}
\left(I_n- A_4^{T}\left(A_4A_4^{T}\right)^{-1}A_4\right)A_3A_2^{\dag}
=
\left(A_4^{\dag}A_5\left(I_n- A_4^{T}\left(A_4A_4^{T}\right)^{-1}A_4\right)\right)^T.
 \end{subnumcases}
\end{theorem}

\begin{proof}
Let the dual $r$-rank decomposition of $\widehat{A}$ exist,
then the DMPGI of $\widehat{A}$ exists.
Let $\widehat{A}=\widehat{A_1}\widehat{A_2}$
be a dual $r$-rank decomposition of $\widehat{A}$
where $\widehat{A_1}=A_2+\epsilon A_3$,
$A_i(i=2,3)\in \mathbb{R}_{n\times r}$,
$\widehat{A_2}=A_4+\epsilon A_5$
and
$A_i(i=4,5)\in \mathbb{R}_{r\times n}$.

$"\Rightarrow"$:
By applying  (\ref{20220216-2})
and the full rank decomposition of $A_0$ to  (\ref{3.5-0-1}),
we have (\ref{2.23}).

By applying  (\ref{20220216-2}) and
$A_1=A_2A_5+A_3A_4$,
we get
\begin{align}
 \label{20220216-4}
\left(I_n- A_4^{T}\left(A_4A_4^{T}\right)^{-1}A_4\right)A_2A_5A_0^\dag
&=
\left(I_n-A_2\left(A_2^{T}A_2\right)^{-1}A_2^{T}\right)A_2A_5A_0^\dag
=0,
\\
\nonumber
\left(I_n- A_4^{T}\left(A_4A_4^{T}\right)^{-1}A_4\right)A_3A_4A_0^\dag
&=
\left(I_n- A_4^{T}\left(A_4A_4^{T}\right)^{-1}A_4\right)A_3A_4A_4^{\dag} A_2^{\dag}
 \\
 \label{20220216-5}
&=
\left(I_n- A_4^{T}\left(A_4A_4^{T}\right)^{-1}A_4\right)A_3 A_2^{\dag} ,
\end{align}
and
\begin{align}
\nonumber
\left(I_n-A_0^\dag A_0\right)A_1A_0^\dag
&=
\left(I_n- A_4^{T}\left(A_4A_4^{T}\right)^{-1}A_4\right)A_1A_0^\dag
\\
\nonumber
&=
\left(I_n- A_4^{T}\left(A_4A_4^{T}\right)^{-1}A_4\right)\left(A_2A_5+A_3A_4\right)A_0^\dag
\\
 \label{20220216-1}
&=
\left(I_n- A_4^{T}\left(A_4A_4^{T}\right)^{-1}A_4\right)A_3 A_2^{\dag} .
\end{align}

In the same way, we have
\begin{align}
\label{20220216-6}
A_0^\dag A_1\left(I_n-A_0^\dag A_0\right)
=A_4^{\dag}A_5\left(I_n- A_4^{T}\left(A_4A_4^{T}\right)^{-1}A_4\right).
\end{align}
From (\ref{20220216-1}), (\ref{20220216-6}) and (\ref{3.5-0-2}),
it follows  that we get (\ref{2.24}).

$"\Leftarrow"$:
Conversely,
if the equation  (\ref{2.23}) holds,
by applying the full rank decomposition of $A_0$,
it is easy to check that  $A_0 A_0^\dag =A_0^\dag A_0$, that is (\ref{3.5-0-1}).
Furthermore, let   (\ref{2.23}) and (\ref{2.24}) hold simultaneously.
 Because $A_0$ is EP,
$\left(I_n- A_4^{T}\left(A_4A_4^{T}\right)^{-1}A_4\right)A_2A_5A_0^\dag=0$ and
$\left(I_n- A_4^{T}\left(A_4A_4^{T}\right)^{-1}A_4\right)A_3A_4A_0^\dag=\left(I_n- A_4^{T}\left(A_4A_4^{T}\right)^{-1}A_4\right)A_3 A_2^{\dag}$.
Therefore,
we get  that
$$
\left(I_n- A_4^{T}\left(A_4A_4^{T}\right)^{-1}A_4\right)A_3 A_2^{\dag}
=\left(I_n-A_0^\dag A_0\right)A_1A_0^\dag.$$
In   the same way,
we have $A_4^{\dag}A_5\left(I_n- A_4^{T}\left(A_4A_4^{T}\right)^{-1}A_4\right)
=A_0^\dag A_1\left(I_n-A_0^\dag A_0\right)$.
It follows from  applying both (\ref{2.24}) and Theorem \ref{3.5-0} that    $\widehat{A}$ is a dual EP matrix.
\end{proof}

\subsection{Dual Penrose Equations}
This subsection
considers   dual Penrose equations
by applying dual $r$-rank decomposition.

\begin{theorem}
Let
 $\widehat{A}\in\mathbb{D}_{m\times n}$,
 $\widehat{A}= A_0+\epsilon A_1$
 and
${\rm rk}(A_0)=r$.
If the dual $r$-rank decomposition of $\widehat{A}$ exists
and $\widehat{A_1}\widehat{A_2}$ is a dual $r$-rank decomposition of $\widehat{A}$,
then
$$(a) \
\widehat{A_2}^{(i)}\widehat{A_1}^{(1)}\in \widehat{A}{\{i\}}(i=1,2,4),
\
(b) \
\widehat{A_2}^{\{1\}}\widehat{A_1}^{(j)}\in\widehat{A}{\{j\}}(i=1,2,3).$$
\end{theorem}
\begin{proof}
(a). When $i=1$,
 both
$\widehat{A_1}\widehat{A_1}^{(1)}$
 and
 $\widehat{A_2}^{(1)}\widehat{A_2}$ are dual idempotent matrices,
then $\widehat{A_1}^{(1)}\widehat{A_1}=I_r$,
and
 $\widehat{A_2}\widehat{A_2}^{(1)}=I_r$,
we get
$$\widehat{A_1}\widehat{A_2}\widehat{A_2}^{(1)}
\widehat{A_1}^{(1)}\widehat{A_1}\widehat{A_2}
=
\widehat{A_1}\widehat{A_2},$$
that is,
 $\widehat{A_2}^{(1)}\widehat{A_1}^{(1)}\in \widehat{A}{\{1\}}$.

When $i=2$,
 $\widehat{A_1}\widehat{A_1}^{(1)}$ is a dual idempotent matrix,
then $\widehat{A_1}^{(1)}\widehat{A_1}=I_r$.
Since
$\widehat{A_2}^{(2)}\widehat{A_2}
\widehat{A_2}^{(2)}=\widehat{A_2}^{(2)}$, 
we get
$$\widehat{A_2}^{(2)}\widehat{A_1}^{(1)}
\widehat{A_1}\widehat{A_2}\widehat{A_2}^{(2)}\widehat{A_1}^{(1)}
=
\widehat{A_2}^{(2)}\widehat{A_1}^{(1)},$$
that is,
 $\widehat{A_2}^{(2)}\widehat{A_1}^{(1)}\in \widehat{A}{\{2\}}$.

When $i=4$,
$\widehat{A_1}\widehat{A_1}^{(1)}$ is a dual idempotent matrix,
then $\widehat{A_1}^{(1)}\widehat{A_1}=I_r$,
we get
$$\widehat{A_2}^{(4)}\widehat{A_1}^{(1)}\widehat{A_1}\widehat{A_2}
=\widehat{A_2}^{(4)}\widehat{A_2}
=\left(\widehat{A_2}^{(4)}\widehat{A_2}\right)^{T}
=\left(\widehat{A_2}^{(4)}\widehat{A_1}^{(1)}\widehat{A_1}
\widehat{A_2}\right)^{T},$$
that is,
$\widehat{A_2}^{(4)}\widehat{A_1}^{(1)}\in \widehat{A}{\{4\}}$.

(b) When $i=1$,
both
$\widehat{A_1}\widehat{A_1}^{(1)}$
and
$\widehat{A_2}^{(1)}\widehat{A_2}$ are dual idempotent matrices,
then
$\widehat{A_1}^{(1)}\widehat{A_1}
=I_r,\ \widehat{A_2}\widehat{A_2}^{(1)}=I_r$,
we get
 $$\widehat{A_1}\widehat{A_2}\widehat{A_2}^{(1)}
 \widehat{A_1}^{(1)}\widehat{A_1}\widehat{A_2}
 =\widehat{A_1}\widehat{A_2},$$
that is,
 $\widehat{A_2}^{(1)}\widehat{A_1}^{(1)}\in \widehat{A}{\{1\}}$.

When $i=2$,
$\widehat{A_2}^{(1)}\widehat{A_2}$ is a dual idempotent matrix,
then $\widehat{A_2}\widehat{A_2}^{(1)}=I_r$,
and
since
$\widehat{A_1}^{(2)}\widehat{A_1}\widehat{A_1}^{(2)}=\widehat{A_1}^{(2)}$,
we get
$$\widehat{A_2}^{(1)}\widehat{A_1}^{(2)}
\widehat{A_1}\widehat{A_2}\widehat{A_2}^{(1)}\widehat{A_1}^{(2)}
=
\widehat{A_2}^{(1)}\widehat{A_1}^{(2)},$$
that is,
$\widehat{A_2}^{\{1\}}\widehat{A_1}^{\{2\}}\in\widehat{A}{\{2\}}$.

When $i=3$,
$\widehat{A_2}^{(1)}\widehat{A_2}$ is a dual idempotent matrix,
then $\widehat{A_2}\widehat{A_2}^{(1)}=I_r$,
we get
$$\widehat{A_1}\widehat{A_2}\widehat{A_2}^{(1)}\widehat{A_1}^{(3)}
=\widehat{A_1}\widehat{A_1}^{(3)}
=\left(\widehat{A_1}\widehat{A_1}^{(3)}\right)^{T}
=\left(\widehat{A_1}\widehat{A_2}\widehat{A_2}^{(1)}
\widehat{A_1}^{(3)}\right)^{T},$$
that is,
$\widehat{A_2}^{\{1\}}\widehat{A_1}^{\{3\}}\in\widehat{A}{\{3\}}$.
\end{proof}



\end{document}